\documentclass[a4paper]{amsart}

\usepackage[USenglish]{babel}
\usepackage{amssymb,amscd}
\usepackage{amsxtra}
\usepackage{amsfonts} 
\usepackage{amsthm}
\usepackage{enumerate}
\usepackage{psfrag}
\usepackage{enumerate}
\usepackage{microtype}
\usepackage{todonotes}
\usepackage[all]{xy}
\usepackage[pdftitle={Bounded negativity of Shimura curves},
	    pdfauthor={Martin M{\"o}ller and Domingo Toledo},
	    linktocpage=true]{hyperref}

\newlength{\halfbls}

\newtheorem{theorem}{Theorem}[section]
\newtheorem{prop}[theorem]{Proposition} 
\newtheorem{cor}[theorem]{Corollary}

\newtheorem{lemma}[theorem]{Lemma}

\newtheorem*{defn}{Definition}

\theoremstyle{remark}
\newtheorem*{remark}{Remark}

\newcommand{\BB}{\mathbb{B}}

\newcommand{\PP}{\mathbb{P}}

\newcommand{\HH}{\mathbb{H}}

\newcommand{\sS}{\mathbb{S}}

\newcommand{\QQ}{\mathbb{Q}}

\newcommand{\RR}{\mathbb{R}}
\newcommand{\CC}{\mathbb{C}}

\newcommand{\NN}{\mathbb{N}}

\newcommand{\BaBo}{{\rm BB}}

\newcommand{\SL}{\mathrm{SL}}

\newcommand{\SO}{\mathrm{SO}}

\newcommand{\SU}{\mathrm{SU}}

\newcommand{\U}{\mathrm{U}}

\newcommand{\tX}{\widetilde{X}}

\renewcommand{\Im}{\mathrm{Im\,}}       %Imaginaerteil

\DeclareMathOperator{\vol}{vol}

\def\be{\begin{equation}}   \def\ee{\end{equation}}     
\def\bes{\begin{equation*}}    \def\ees{\end{equation*}}
\def\ba{\be\begin{aligned}} \def\ea{\end{aligned}\ee}   
\def\bas{\bes\begin{aligned}}  \def\eas{\end{aligned}\ees}

\newcommand{\frakO}{\mathfrak{O}}

\newcommand{\ve}{\varepsilon}

\def\be{\begin{equation}}   \def\ee{\end{equation}}     \def\bes{\begin{equation*}}    \def\ees{\end{equation*}}
\def\ba{\be\begin{aligned}} \def\ea{\end{aligned}\ee}   \def\bas{\bes\begin{aligned}}  \def\eas{\end{aligned}\ees}

\makeatletter
\renewcommand{\subsubsection}{\@startsection{subsubsection}{2}%
        {\z@}{-3.25ex plus -1ex minus-.2ex}{-1em}{\bf}}
\makeatother

\title[Self-intersection numbers of Shimura curves]{Bounded negativity of self-intersection numbers of Shimura curves on Shimura surfaces}

\thanks{The first named author is partially supported
by the ERC-StG 257137.  
The second named author is partially supported by  Simons Foundation Collaboration Grant 208853}
\author{ Martin M{\"o}ller and Domingo Toledo}
\date{\today}

\begin{document} 

\begin{abstract}
Shimura curves on Shimura surfaces have been a candidate for
counterexamples to the bounded negativity conjecture. We prove
that they do not serve this purpose: there are only finitely
many whose self-intersection number lies below a given bound.
\par
Previously, this result has been shown in \cite{negative} 
for compact Hilbert modular surfaces using the Bogomolov-Miyaoka-Yau inequality. 
Our approach uses equidistribution and works uniformly for all
Shimura surfaces.
\end{abstract}

\maketitle

%%%%%%%%%%%%%%%%%%%%%%%%%%%%%%%%%%%%%%%%%%%%%%%%%%%%%%%%%%%%%
\section*{Introduction}
%%%%%%%%%%%%%%%%%%%%%%%%%%%%%%%%%%%%%%%%%%%%%%%%%%%%%%%%%%%%%

Let $X$ be a  %\todo{is there a standard terminology for this?} 
Shimura surface not isogeneous to a product, i.e.\ an algebraic surface which is the
quotient of a two dimensional Hermitian symmetric space $G/K$ by an irreducible arithmetic lattice in $G$. 
The aim of this note is to show that Shimura curves on such a Shimura surface do not 
provide a counterexample to the bounded negativity conjecture. More precisely we show:
\par
\begin{theorem} \label{thm:intro_main}
For any Shimura surface $X$ not isogeneous to a product 
and for any real number  $M$ there are only finitely many compact Shimura curves 
$C$ on  $X$ with $C^2 < M$. 
\end{theorem}
\par
The bounded negativity conjecture claims that for any smooth projective
algebraic surface $X$  there is a positive constant $B$ so that  for any   
irreducible curve $C$ on $X$  the self-intersection $C^2 \ge - B$. 
We emphasize that the above theorem does not decide 
the validity on any Shimura surface, as there could exist
non-Shimura curves with arbitrarily negative self-intersection.
\par
There are two possibilities for the uniformization of $X$. The first case
are Shimura surfaces uniformized by $\HH^2$.  In this case 
$G  = \SL_2(\RR)^2$ and the surfaces are called {\em quaternionic
Shimura surfaces} if $\Gamma$ is cocompact and {\em Hilbert modular
surfaces} if $\Gamma$ has cusps. The second case are  Shimura surfaces 
uniformized by the complex $2$-ball $\BB^2$. In this case 
$G= \SU(2,1)$  and the surfaces are called {\em Picard modular
surfaces}. There are compact and non-compact Picard modular surfaces.
The assumption on the Shimura surface is necessary, since the theorem is certainly 
false in the product situation e.g.\ for $X = X(d) \times X(d)$
a product of modular curves or a finite quotient of such a surface: the fibre classes
give infinitely many curves with self-intersection zero.
\par
While only the case of compact $X$ is relevant to the bounded negativity conjecture, the proofs for non-compact $X$  are the same. When both $X$ and the curves $C$ are allowed to have cusps the proper formulation is needed, see Theorem \ref{thm:main_noncompact}.
\par
Theorem~\ref{thm:intro_main} was proven for compact Shimura surfaces 
uniformized by $\HH^2$ in \cite{negative}. The methods there, based on
the logarithmic Bogomolov-Miyaoka-Yau inequality, do not extend
to the ball quotient case. Here we give a uniform treatment of
both cases based on equidistribution results. As in loc.~cit.\ 
we obtain as a consequence:
\par
\begin{cor}
There are only finitely many Shimura curves on $X$ that are
smooth.
\end{cor}
\par
Intersection numbers of Shimura curves are known to appear as coefficients
of modular forms and coefficients of modular forms are known to grow.
This, however, does not directly give a method to prove 
Theorem~\ref{thm:intro_main}, 
since in these modularity statements (\cite{HZ}, \cite{Kudla}) the
Shimura curves are packaged to reducible curves $T_N$ with an unbounded
number of components as $N \to \infty$, while
the statement here is for every individual Shimura curve.
\par
{\bf Acknowledgements.}\ This note is the result of an MFO-miniworkshop
``The bounded negativity conjecture''. We thank the organizers of this workshop,
and the MFO administration for simultaneously organizing a miniworkshop
``K\"ahler groups'', that brought the two authors together. We thank Alex Eskin
for discussing details on equidistribution results, and Madhav Nori and Matthew Stover for discussions on algebraic group details.  The second named author thanks Misha Gromov for a conversation in 1990 on the idea that Shimura curves approximate the K\"ahler class.
\par

%%%%%%%%%%%%%%%%%%%%%%%%%%%%%%%%%%%%%%%%%%%%%%%%%%%%%%%%%%%%%
\section{Shimura curves on Shimura surfaces not isogeneous to a product}
%%%%%%%%%%%%%%%%%%%%%%%%%%%%%%%%%%%%%%%%%%%%%%%%%%%%%%%%%%%%%

An {\em  Shimura surface  not isogeneous to a product} is a connected algebraic surface that can
be written as a quotient $X = \Gamma \backslash G /K$, where $G = G_\QQ(\RR)$
is the set of $\RR$-valued points in a connected semisimple $\QQ$-algebraic group $G_\QQ$, 
where $K \subset G$ is a maximal compact subgroup and where $\Gamma$ is 
an irreducible arithmetic lattice in $G$. Here  a lattice is called irreducible 
if it does not have a finite index subgroup that splits as a product of two lattices.
\par
Our geometric definition of Shimura varieties differs from the arithmetic
literature on this subject where Shimura varieties are typically not connected.
It is the point of view of the bounded negativity conjecture that requires
to deal with irreducible components of the objects in question. Note that
we do not require $\Gamma$ to be a congruence subgroup either.
\par
\begin{defn}
A {\em Shimura curve} $C$ is an algebraic curve in $X$ which is 
given as follows. There exists a $\QQ$-algebraic group $H_\QQ$ containing
an arithmetic lattice $\Delta$ and admitting a $\QQ$-morphism $\tau: H_\QQ \to G_\QQ$ 
such that $\tau(\Delta) \subset \Gamma$, such that $\tau$-preimage of a
maximal compact subgroup $K \subset G_\RR$ is a maximal compact subgroup $K_H \subset H= H_\QQ(\RR)$
and $C = \Delta \backslash H / K_H$.
\end{defn}
\par
The aim of this section is to compile the list of possible constructions of
Shimura surfaces that contain infinitely many Shimura curves and the possible
pairs $(G_\QQ, H_\QQ)$. This will be used in the equidistribution theorem in the
next section. More precisely, we need that all Shimura curves can be generated
as the orbit of a fixed subgroup. For this purpose we write $G = G_0 \times W$ with $W$ compact and $G_0$ without compact factors. There is a corresponding
decomposition of the compact subgroup $K = K_0 \times W$ and also
for the Shimura curve $H = H_0 \times W_H$ and $K_H = K_{H,0} \times W_{H}$.
\par
It turns out that there are only two possibilities
for $G_0$ and for each of them, we can construct all Shimura curves as follows.
\par
\begin{prop}\label{prop:constructcurves}
For a given  Shimura surface  $X = \Gamma \backslash G_0 /K_0 =   \Gamma \backslash G/K$  
not isogeneous to a product there exists subgroup $H_0 \cong \SL_2(\RR)$ of $G_0$ such that all Shimura curves arise as $C = \Gamma \backslash \Gamma g H_0/K_{H_0}$ for some $g \in G_0$.
\end{prop}
\par 
We start with the possibilities for $G_0$. 
There are only two hermitian symmetric domains of dimension two. This leads
to the following two cases, as in the introduction. In each case we give a description of the possible Shimura surfaces.  Here, and elsewhere, the description of the algebraic groups in question will always be given only up to central isogeny.
\smallskip
\par
{\bf Case One, $G_0= \SL_2(\RR)^2$:} There two possibilities. Either 
$G$ is  the set of $\RR$-points of the $\QQ$-algebraic group 
$G_\QQ = {\rm Res}_{F/\QQ}(\SL_2(A))$ for a quaternion algebra $A$ over
a totally real field $F$ which is unramified at exactly
two infinite places of $F$ or $G$ is the product  %\todo{could you give a reference?}
${\rm Res}_{F/\QQ}(\SL_2(A_1)) \times {\rm Res}_{F/\QQ}(\SL_2(A_2))$
for two quaternion algebras, each unramified at exactly at one infinite place.
	For the proofs, first remark that these give $F$-forms of $\SL_2(\RR)^2$, see, e,g, \cite[IV.1]{Vigneras}. That these are the only possibilities follows from the classifcation of algebraic groups \cite{Tits}.  In more detail, the procedure of \cite[\S 3.1]{Tits} reduces the problem to the classification of $F$-forms of $\SL_2$.   The description in \cite[III.1.4]{Serre} of the $F$-forms of $\SL_2$ in bijective correspondence with quaternion algebras over $F$ gives the above description of the algebraic groups.    In both cases, the maximal compact
subgroup $K$ in $G$ is $\SO_2(\RR)^2$ times the compact factors of $G_\RR$. 
%\par
%We remark that the description in \cite{Tits} gives, in his notation, simple factors of type $^1A_{1,r}^{(d)}$ with $d(r+1) = 2$ and leads to different but equivalent formulation of the classification: either ${\rm Res}_{F/\QQ}(\SL_1(D))$ where $D$ is a quaternion division algebra ($r=0, d=2$, a co-compact group), or ${\rm Res}_{F/\QQ}(\SL_2(F))$ ($r=1, d=1$, not co-compact.  if $F = \QQ$ this is he product case, if $F$ is a quadratic extension of $\QQ$ it is the irreducible case.
\par 
In the product case, all lattices are reducible, so we can discard this case
in view of our irreducibility hypothesis on $X$. In the remaining case, 
in order obtain an arithmetic lattice $\Gamma \subset G$ one has
to fix an order $\frakO \subset A$ and let $\frakO^1 \subset \frakO$
be the elements of reduced norm $1$. Then $\Gamma$ is the image
in $G$ of a group commensurable to $\frakO^1$.
See e.g.~\cite{Vigneras} for more details.
\smallskip
\par
{\bf Case Two, $G_0=\SU(2,1)$:} In this case, the underlying 
$\QQ$-algebraic group is $G_\QQ = {\rm Res}_{F_0/\QQ}(G_{F_0}))$, and, from the classification of algebraic groups (over number fields), see \cite{Tits,PlatRap},  we see that, in the notation of p. 55 of  \cite{Tits}, $G_{F_0}$ must be of type $^2 A^{(d)}_{2,r}$,  where $d|3$, $d\ge 1$, $2rd\le 3$.  In other words, $G_{F_0} = SU(h)$ where 
$h$ is a hermitian form constructed as follows. 
Start with a totally real field $F_0$  and take a totally complex quadratic
extension $F/F_0$, i.e.\ $F$ is a CM field.  Then take a central simple division algebra $D$ of degree $d$ (hence dimension $d^2$) over $F$, with center $F$ and involution $\sigma$ of the second kind (not the identity on $F$), and a hermitian form $h$ on $D^{3/d}$ so that $h$ is isotropic at one real place of $F_0$ and definite at all other real places (equivalently, isotropic at one conjugate pair of complex places of $F$, definite at all other pairs).
\par
Thus there are two \lq\lq types" corresponding to the two possibilities   $d= 1$ or $d = 3$:
\par
The {\em first type} means that $d =1$.  Then  $D = F$ and  $h$ is a hermitian form on $F^3$ that
is definite except for one pair of places of $F$, interchanged
by complex conjugation. Then $\SU(h)$ is indeed a $F_0$-algebraic
group and the set of $\RR$-valued points of ${\rm Res}_{F_0/\QQ}(\SU(h))$
equals $G$ up to compact factors. The compact
subgroup $K$ in $G$ is $S(U(2)\times U(1))$ times 
the compact factors of $G_\RR$.  
Arithmetic lattices $\Gamma$ {\em of the first type} are obtained by fixing an
order $\frakO \subset F$ and taking $\Gamma$  commensurable to 
 $G \cap \SL_3(\frakO)$.   The integer $r$ above satisfying $2rd\le 3$ is the $F_0$-rank of $G_{F_0}$, or the dimension of the maximal isotropic subspace  of $h$ in $F^3$.  The lattice is co-compact if and only if $r = 0$, and $r=1$ forces $F_0 =\QQ$.
\par
The {\em second type} means that $d=3$, thus $D$ is   central simple division 
algebra of degree $3$  (dimension $9$) over $F$ with an involution \lq\lq of the second kind".  The lattices $\Gamma$ are obtained by fixing an order $\frakO\subset D$ and taking $\Gamma$ commensurable with $G\cap SL(D)$.  Observe that in this case the inequality $2rd\le 3$ forces $r = 0$ and therefore $\Gamma$ is always co-compact.  We will see that lattices of the second type do not have any Shimura curves, so we will not need to consider them.

\par
\smallskip
{\bf Shimura curves in $X$ for $G_0 = \SL_2(\RR)^2$}. 
The Shimura curves in $X$ are totally geodesic complex curves in $X$, 
so they are projections to $X$ of totally geodesic holomorphic disks  
$\HH  \subset\HH^2$, which in turn are orbits of embeddings of 
$\SL_2(\RR) \subset \SL_2(\RR)^2$.  It is well known that, up to biholomorphic isometries, there are only two classes of such disks:  factors and diagonals.   By the irreducibility hypothesis, 
the inclusion into one factor does not come from a morphism of the 
underlying $\QQ$-algebraic groups.   
So $H_0 \subset G_0$ has to be the diagonal embedding, proving 
Proposition~\ref{prop:constructcurves} in this case. In fact, the 
possible embeddings are discussed in great detail in \cite{vandergeer88}
for Hilbert modular surfaces and in \cite{granath} for quaternionic
Shimura surfaces.
\par
\smallskip
{\bf Shimura curves in $X$ for $G_0 = \SU(2,1)$}.   
Fix a Shimura surface $X$ obtained by choosing 
$F_0, F,d, D, \sigma,   h ,\frakO \subset D, \Gamma$.  The Shimura curves, being totally geodesic complex curves, are projections to $X$ of orbits in the universal cover of subgroups $H\subset G_0$, all isomorphic to $SU(1,1)$ and standardly embedded in $SU(2,1)$.   The image in $X$ of an $H$-orbit is a Shimura curve if and only if $H\cap \Gamma$ is a lattice in $H$.  This happens if and only if $H$ is defined over $F_0$, meaning that the underlying algebraic group $G_{F_0}$ contains an $F_0$-subgroup $H_{F_0}$ so that, if $\iota: F_0\to \RR$ is the embedding of $F_0$ with group of real points $G_{F_0,\iota}(\RR)$ isomorphic to $G_0$, the inclusion  $H_{F_0,\iota}(\RR)\subset G_{F_0,\iota}(\RR)$ agrees with $H\subset G_0$.   There are two cases:
\smallskip
\par
{\bf No Shimura curves in Shimura surfaces of the second type}: The group $SU(h)$, for $h$ a hermitian form on a central simple division algebra $D$ over $F$ of degree three as above, has no subgroup $H_{F_0}$ defined over $F_0$ with $H_{F_0}(\RR) = SU(1,1)$ standardly embedded in $SU(h)(\RR) = SU(2,1)$.
\par
This is well-known to experts, but we do not know a reference (but see \cite[Corollary~4.2]{GaribaldiGille} for a more general result).  Matthew Stover kindly 
communicated the following proof.

Let $F_0,F,D,\sigma$ be as above. The $D$-valued hermitian form $h$ can be taken to be $h(x,y) = \sigma(x) y$ and the group of $F_0$-points of the $F_0$-group in question is
$$
SU(D,\sigma) (F_0) = \{x\in D: \sigma(x)x = e, \ Nrd(x) = 1\}\subset D.
$$
which gives us an $SU(2,1)$ as follows:  choose an embedding $F\to \CC$, use it to form $D\otimes_F\CC$ which becomes isomorphic to the algebra $M(3,\CC)$ of three by three complex matrices, under an isomorphism (unique up to conjugation by Skolem-Noether), which takes $\sigma$ to conjugate-transpose with respect to a hermitian form $h'$.  Whenever all choices can be made so that $h'$ has signature $(2,1)$  the group of real points of $SU(D,\sigma)$ becomes the standard $SU(2,1)$.  The signature of the hermitian form $h'$ depends just on $D,\sigma$ and the embedding $F\to\CC$.
\par 
Note that the $F$-algebra $D$ is embedded in the algebra  $M(3,\CC)$ by $x \to x\otimes 1$.   The $F$-vector subspace of $M(3,\CC)$ generated by the subset $SU(D,\sigma)(F_0)$ is easily seen to be a $\sigma$-stable subalgebra of $M(3,\CC)$ contained in the division algebra $D$, hence it is itself a division algebra, and easily seen to equal $D$.   Suppose $H_{F_0}$ is an $F_0$-subgroup of $SU(D,\sigma)$ so that the corresponding inclusion of real points is a standard embedding of $SU(1,1)$ in $SU(2,1)$, all inside $M(3,\CC)$, and let $V$ be the $F$-vector subspace of $M(3,\CC)$ generated by the $F_0$-points of $H_{F_0}$.  This is a non-commutative division subalgebra of $D$, and it must be a proper subalgebra because $V\otimes_F\CC$ is a proper subspace of $D\otimes_F\CC = M(3,\CC)$.  Since $D$ has degree $3$, it has no proper non-commutative $F$-subalgebras, so such subgroups cannot exist.
\par
\smallskip

%With $F_0$ and $F$ as above, the groups $SU(h)$ for $h$ a hermitian form on a central simple division algebra over $F$ with involution of the second kind and of {\em prime degree} contain no proper semi-simple $F_0$ - subgroups, see Corollary~4.2 of \cite{GaribaldiGille}. Alternatively, we can refer to \cite[Theorem~3.2 and page~8]{BlasiusRogawski}: If $\Gamma\subset SU(h)$ is a torsion-free  lattice of the second type,  it is known that for any congruence subgroup 
%$\Gamma'\subset\Gamma$ the Picard number of $X'$ is always one.  
%In particular, for any such cover $X'$ every curve has positive self-intersection. 
%So if Shimura curves $C\subset X$ existed, then embedded ones would also exist, possibly passing to some cover $X'$, giving a curve of negative self-intersection.
%\par
\par
{\bf Classification of Shimura curves in Shimura surfaces of the first type}:  In this case there are always infinitely many Shimura curves.   We continue the same notation, extend the hermitian form $h$ on $L^3$ to $\CC^3$ and   interpret the unit ball 
$G_0/K_0 \cong \BB^2\subset\PP^2$ as  the collection of $h$-negative lines in $\CC^3$.
The Shimura curves in $X$ arise as the quotient of  totally geodesic disks
$\BB^1 \subset\BB^2$ and such disks are in bijective correspondence with the $h$-positive lines.  Namely, an 
 $h$--positive line $l$ determines the hermitian space  $(\ell^\perp,h|_{ l^\perp})$ 
of signature $(1,1)$, and the corresponding space of negative lines $\BB^1_l\subset \BB^2$.  All geodesic disks arise this way.  The groups  $G_\ell$, 
the stabilizer of $\ell$  (isomorphic to $U(1,1)$) and the subgroup $H_l$ fixing $l$ pointwise (isomorphic to $SU(1,1)$) act on  $(\ell^\perp,h|_{ l^\perp})$ and 
  $\BB^1_\ell$, both actions being transitive on $\BB^1_l$. 
The disk $\BB^1_\ell$ projects to a Shimura curve in $X$.  
if and only if  $H_\ell\cap \Gamma$ a  lattice in $H_\ell$, in turn: 
\par
\begin{lemma} \label{le:closedgeodesic}
The group $H_\ell\cap\Gamma$ is a lattice in $H_\ell$ if and only if $\ell$ 
is an $F$-rational line, that is,  $\ell\cap  F^3\ne \{0\}$. 
\end{lemma}
\par
\begin{proof}
Let $v\in \CC^3$ be a basis vector for $\ell$, and suppose that  
$\Gamma_\ell = H_\ell\cap\Gamma$ is a lattice in $H_\ell$.  Since 
$\Gamma_\ell$  leaves $\ell$ stable, $v$ is an eigenvector for all  
$\gamma\in H_\ell \cap\Gamma$, in other words, there is a homomorphism
$\lambda :\Gamma_\ell\to\U(1)\subset \CC^*$ so that 
$\gamma (v) = \lambda (\gamma ) v$ for all $\gamma\in \Gamma_l$.  
Since $\Gamma_\ell$ leaves $\ell^\perp$ invariant, the remaining 
eigenvectors of any $\gamma\in \Gamma_\ell$ lie in $\ell^\perp$.  
Since the action of $H_\ell$ on $l^\perp$ is isomorphic to the 
standard action of $\SU(1,1)$ on $\CC^2$ and $\Gamma_\ell$ is a 
lattice in $H_\ell$, the commutator subgroup of $\Gamma_l$ must 
contain hyperbolic elements.   Fix such an element $\gamma$.  
Then $\lambda(\gamma) = 1$ and  the remaining eigenvalues of $\gamma$ are of absolute value $\ne 1$. Therefore $1$ is a simple eigenvalue of $\gamma$, thus the space of solutions of  $\gamma(v) = v$  is  an $F$-rational line as asserted. 
\par
For the converse, suppose that $\ell$ is a rational line, and let 
$v\in\frak{O}^3$ be a primitive vector which is a basis for $\ell$.  
 Let  $M_0  = \frakO v $ and $M_1 = v^\perp \cap\frak{O}^3$ and let 
$M = M_0\oplus M_1$.   Then $M$ is an $\frakO$-submodule of finite 
index in $\frakO^3$.  Consequently, $\Gamma$ is commensurable with $\Gamma' = \{\gamma\in\SU(h,\frakO):\gamma(M) = M\}$ and $\Gamma\cap H_l$ is commensurable with $\Gamma'_v = \{\gamma\in \Gamma' : \gamma(v) = v\}$, which is a lattice in the group $H_\ell = H_v = \{g\in G: g(v) = v \}$, a group defined over $F_0$, 
and isomorphic 
(over $F_0$) to $\SU(h|_{M_1\otimes F})$. This group in turn is isomorphic 
over $\RR$ to $\SU(1,1)$.  Thus $\Gamma\cap H_\ell$ is a lattice in 
$H_\ell$ and we obtain a Shimura curve associated to the $\QQ$-group 
${\rm Res}_{F_0/\QQ}(\SU(h|_{M_1\otimes F}))$.
\end{proof}
\par
\emph{End of proof of Proposition~\ref{prop:constructcurves}} : Choose 
an orthogonal basis $v_1,v_2,v_3$ for $\frakO^3$ where $h(v_i) = a_i \bar a_i >0$ for $i = 1,2$,  $h(v_3) = -a_3 \bar a_3<0$ and $v_1\in \ell$.  
Let $e_1,e_2,e_3$ be the standard basis for $\CC^3$, let $H = H_{e_1}\subset G$ be the subgroup, isomorphic to $\SU(1,1)$ that fixes $e_1$, and 
let $g\in G$ be the linear transformation that takes $e_i$ to $v_i/a_i$.   
Then $g H g^{-1} = H_\ell$, therefore $H_\ell$ is as asserted in 
Proposition~\ref{prop:constructcurves} 
\par
\begin{remark}
From Lemma~\ref{le:closedgeodesic} we see that the collection of Shimura curves in $X$ is parametrized by the $\Gamma$-equivalence classes of primitive positive vectors in $\frakO^3$, that is, primitive vectors $v\in \frakO^3$ with $h(v) >0$.  The collection of these equivalence classes is commensurable with $SU(h,F)\backslash\PP(F^3)^+$, where $\PP(F^3)^+$ denotes the space of $h$-positive lines in $F^3$.  The class of $h(v)$ gives is a well-defined function $h:\PP(F^3)\to F_0^*/N_{F/F_0}(F^*)$, the \emph{norm residue group}.  It can be checked that the class of $h(v)$ is a commensurability invariant and that it takes on infinitely many values, hence we get an infinite number of commensurability classes of subgroups of $SU(1,1)$.  Observe that the  matrix of the conjugating element $g$ of Lemma \ref{le:closedgeodesic} has entries in the finite field extension $F(a_1, a_2,a_3)$ of $F$.
\end{remark}
\par
\medskip
The compact factors of $G$, necessary for the $\QQ$-structure in the
definition of a Shimura surface, play no role in the sequel. We thus
simplify notation and write from now on $G$ for $G_0$ and $H$ for $H_0$.
\par
\medskip
{\bf Elliptic elements and cusps.} The bounded negativity conjecture (BNC) 
originally is a question for smooth compact (projective) surfaces. If $\Gamma$ is 
cocompact and torsion free, Shimura surfaces as defined above fall into the
scope of this conjecture and the results in the introduction need no explanation.
\par
Any arithmetic lattice contains a neat subgroup of finite index. Such subgroups are
in particular torsion free.  As quotients by a finite group, the Shimura surfaces 
come with a ($\QQ$-valued) intersection theory. The BNC can be extended
to such surfaces, and Theorem~\ref{thm:intro_main} needs no further explanation.
\par
If $\Gamma$ is cofinite but not cocompact, our proof of Theorem~\ref{thm:intro_main}
gives a statement about the self-intersection number of the cohomology class
of the Shimura curve projected to the complement of the cusp resolution 
cycles, as we will now explain.
\par
We may suppose that $\Gamma$ is a neat subgroup. Let $X^\BaBo$ be the minimal 
(Baily-Borel) compactification of $X = \Gamma \backslash G /K$. Since
$X$ is not isogenous to a product, $X^\BaBo \setminus X$ has codimension two and
hence $H^2_c(X,\QQ) \cong H^2(X^\BaBo,\QQ)$.
Let $\pi: Y \to X^\BaBo$ a (minimal) smooth resolution of the singularities at
the cusps and $j: X \to Y$ the inclusion. We claim that 
\begin{equation} \label{eq:H2dec}
H^2(Y,\QQ) = \pi^*H^2(X^\BaBo,\QQ) \oplus B,
\end{equation}
where $B$ is the subspace spanned by cusp resolution curves. Moreover, 
the direct sum is orthogonal and  the intersection form on $B$ is 
negative definite. This implies that
the sum decomposition is compatible with Poincar\'e\ duality and
this will make the arguments in Section~\ref{sec:current_of_integraion} 
work in the non-compact case, too,  see Theorem \ref{thm:main_noncompact}.
\par
Our claims are stated for the Hilbert modular case in \cite[Section~II.3 and
Section~VI.1]{vandergeer88}). In the case of a ball quotient a neighborhood $W$
of the cusps in $Y$ is disjoint union of disc bundles over tori, each sitting inside a line bundle
of negative degree. It suffices show that $H_2(Y,\QQ) = H_2(W,\QQ) \oplus {\rm Im}
(j_*:H_2(X,\QQ) \to H_2(Y,\QQ))$ and then apply duality. By Meyer-Vietoris, it
suffices to show that $H_1(W\cap X,\QQ) \to H_1(W,\QQ) \oplus H_1(X,\QQ)$ is injective. This
holds true, since the inclusion of a circle bundle into the corresponding disc bundle
induces an injection the level of $H_1(\ \ ,\QQ)$.
\par
 We remark that the BNC (and intersection numbers in general) are 
very sensitive to blowups. We leave it to the reader to investigate 
if Theorem~\ref{thm:intro_main} also holds on $Y$.
\par
\medskip
{\bf Volume normalization.} The Hermitian symmetric space
$G/K$ comes with a K\"ahler $(1,1)$-form $\omega$ that we
normalize, say, so that the associated Riemannian metric
has curvature attains the minimum $-1$. Then $\omega \wedge \omega$ provides
a volume form on $X$ and, consequently, also on the universal covering $\tX$.
We let $\vol(X)$ be the volume
of the Shimura surface. Rescaling by the volume we
obtain a probability measure $\nu_X$ on $X$ induced from the
volume form. 
\par
Shimura curves are totally geodesic subvarieties in $X$.
Consequently, the restriction of $\omega$ is a K\"ahler
form $\omega_C$ on $C$. We let $\vol(C) = \int_C \omega_C$
be the corresponding volume and $\nu_C$ the probability
measure defined by $\omega_C$. 
\par
We need to extend this to the quotients by smaller compact
subgroups.  Let $K' \subset G$ be a compact subgroup and $K'_H = K' \cap H$.
Let $\nu_G$ be the Haar measure on $G$ normalized so that the push-forward to $G/K$ 
gives the above volume form on $\tX$ and such that the fibers have volume one. 
From $\nu_G$ we obtain measures
$\nu_{G/K'}$ on $G/K'$ and finite measures $\nu_{\Gamma \backslash G/K'}$ 
on $X_{K'} = \Gamma \backslash G/K'$ with $\vol(X) = \vol(X_{K'})$.
\par
Similarly we fix a normalization of a Haar measure $\nu_H$ on $H$ 
by requiring that the fibers of $H \to H/K_H$ have volume one 
and that the push-forward to $H / K_H$ is the volume form 
coming from the metric with curvature $-1$, as above.
\par
In this way, given a Shimura curve $C = \Gamma \backslash \Gamma g H/K_H$, 
the push-forward of $\nu_H$ defines a finite measure 
$\nu_{C,K'}$ on the locally symmetric subspaces 
$C_{K'} =  \Gamma \backslash \Gamma g H/K_H'$ inside $X_{K'}$
with  $\vol(C_{K'}) = \vol(C)$.
\par

%%%%%%%%%%%%%%%%%%%%%%%%%%%%%%%%%%%%%%%%%%%%%%%%%%%%%%%%%%%%%
\section{Equidistribution}
%%%%%%%%%%%%%%%%%%%%%%%%%%%%%%%%%%%%%%%%%%%%%%%%%%%%%%%%%%%%%

There are many sources in the literature that deduce equidistribution
for Shimura curves from a Ratner type theorem (notably \cite{ClUlEqui}, 
\cite{UlEquiII}). We need a slightly
stronger equidistribution result, on $\Gamma \backslash G$
or on on $\Gamma \backslash G /K'$ for some (not necessarily
maximal) compact subgroup $K'$ of $G$ rather than on the algebraic
surface $X$. This follows along known lines from Ratner's
result, or rather the version in \cite{EMSunipotent}. We 
give a proof avoiding technicalities on Shimura data
and focussing on the surface case.
\par
The references above contain as special case the following equidistribution
                   \par
\begin{prop} \label{prop:equidist}
Suppose that $X$ is a Shimura surface.
If $(C_n)_{n \in \NN}$ is a sequence of pairwise
different Shimura curves, then $\nu_{C_n} \to \nu_X$ weakly 
as $n \to \infty$. 
\end{prop}
\par
This is a special case of the following stronger result.
\par
\begin{prop} \label{prop:equidistK}
Suppose that $X = \Gamma\backslash G/K$ is a Shimura surface. Let $K' \subset K$ be a closed subgroup,  and let 
$g_n \in G$ be a sequence of points so that  the orbits $g_n H\subset G$
project to pairwise distinct Shimura curves $C_n$ in $X$. Then on $X' = \Gamma \backslash G/K'$
the sequence of probability measures $\nu_{C_n,K'}$ converges weakly to $\nu_{\Gamma \backslash G/K'}$
as $n \to \infty$.
\end{prop}
\par
\begin{cor} \label{cor:vol}
Suppose that $X = \Gamma \backslash G / K$
is a Shimura surface. If $(C_n)_{n \in \NN}$ is a sequence of pairwise
different Shimura curves, then $\vol(C_n) \to \infty$ 
as $n \to \infty$.
\end{cor}
\par
\begin{proof}[Proof of Corollary~\ref{cor:vol}] With the above volume
normalization, it suffices to prove the claim for the lifts of the Shimura
curves $C_n'$ to $X' = \Gamma \backslash G$. We apply the preceding proposition
for $K' =\{e\}$.
Equidistribution implies in particular that Shimura curves are dense, i.e.\
for any finite collection of open sets $U_i$, $i \in I$. there exists $N_0$ such that
for $n > N_0$ the intersection $C_n \cap U_i$ is non-empty for all $i$. Since
$X'$ is foliated by $H$-orbits and $\nu$ is locally the product of $\nu_G$ and
a transversal measure, it suffices to take for $U_i$ sufficiently many  
open sets locally trivializing the foliation $U_i = V_i \times W_i$ with
$V_i$ an $H$-orbit, such that $\nu_H(V_i) = O(1)$ but the transversal measure
of $W_i$ is $O(1/n^2)$. Then we can fit $O(n)$ such sets into $X$ and each
time $C_n$ intersects some $U_i$, it picks up a volume of $O(1)$. 
\end{proof}  
\par 
\begin{proof}[Proof of Proposition~\ref{prop:equidistK}] We first observe that if the Proposition holds for $K' = \{e\}$, then it holds for any other $K'\subset K$.  Namely, under the projection $\pi:X'' = \Gamma\backslash G\to X' = \Gamma\backslash G /K'$ we have, by the volume normalization above, that the push-forward measure satisfy $\pi_*(\nu_{X''}) = \nu_{X'}$ and $\pi_*(\nu_{C_n,e} )= \nu_{C_n,K'}$ .  Thus we will assume $K' = \{e\}$, and write simply $\nu_n'$ for $\nu_{C_n,e}$ and $X'$ for $\Gamma\backslash G$.
\par 
The proof consists of two parts:  1.  Prove that $\nu_n'$ has convergent subsequences $\nu_{n_j}'$.   2. Prove that the limit of any convergent subsequence must be $\nu_{X'}$.
\par
If $\Gamma$ is co-compact, that is, $X'$ is compact, then the space of probability measures on $X'$ is compact in the weak  * topology, so $\nu_n'$ has a convergent subsequence.  If  $X$ is not compact, then a subsequence converges to a measure on the one point  compactification $X'\cup\{\infty\}$, but these measures may \lq\lq escape to infinity", say converge to the delta function at $\infty$.  An example of this  \lq\lq escape of mass" is given in the introduction to \cite{EMSnondiv}.  The main result there  is that there is no escape of mass when the image of $Z(H)$  in $X'$ is compact (where $Z(H)$ is the centralizer of $H$ in $G$

).  More precisely,  compactness of the image of  $Z(H)$  in $X'$  implies (see \cite[Theorem~1.1]{EMSnondiv})
that for every $\ve>0$ there exists a compact subset $W \subset \Gamma
\backslash G$ such that every $H$-orbit gives measure at least
$1-\ve$ to $W$. Hence the sequence $\nu_n'$ indeed converges in the space of probability measures on $X'$.
\par

In our situation $Z(H)$ itself is compact:  it is finite in Case 1 and $U(1)$ in Case 2, thus we always have convergence, thereby proving (1).  (Compactness of $Z(H)$ generally holds for Shimura varieties if one
discards the obvious exception of product situations, see \cite{UlEquiII}.)

\par

To prove (2) we may assume $\nu_n'$ converges weakly to a probability measure $\nu'$, we must prove $\nu' = \nu_{X'}$.  This follows a pattern which is by now standard: (i) use, as in \cite{EMSunipotent}, Ratner's theorem on  unipotent flows to prove that $\nu$ 
is {\em algebraic}, i.e.\ supported on an $L$-orbit of some
connected algebraic group $H \subseteq L \subseteq G$ that intersects
$\Gamma$ in a lattice.  (ii) Prove $L=G$.   We formulate (i) as the following lemma:

\par
\begin{lemma} \label{le:algebraic}
Suppose $\nu_n'$ converges weakly to $\nu'$. Then there exists a closed connected subgroup $L$, $H\subset L\subset G$, such that $\nu'$ is an $L$-invariant measure supported on $\Gamma\backslash\Gamma c L$ for some $c\in G$ and such that  $c^{-1}\Gamma c \cap L$ is a lattice in $L$.  Moreover, there exists a sequence $x_n\in \Gamma g_n H$ converging to $c$ and an $n_0$ such that $c L c^{-1}$ contains the subgroup generated by $x_n H x_n^{-1}$ for $n\ge n_0$.
\end{lemma}
\par

We formulated this lemma following closely the wording of \cite[Proposition~2.1]{EOHecke} (see also \cite[Theorem~1.7]{EMSunipotent}) because it can be proved from \cite[Theorem~1.1]{MS} in same way.  Namely, start from the fact that $\nu_n'$  
 is supported on the $H$-orbit $\Gamma\backslash \Gamma g_n H\subset \Gamma\backslash G$ which is isomorphic to $(g_n^{-1}\Gamma g_n \cap H)\backslash H$ and is $H$-invariant.  Since $g_n^{-1} \Gamma g_n$ is a lattice in $H$, which, in our case, is locally isomorphic to $SL(2,\RR)$,  we can  choose a unipotent one-parameter subgroup $u(t)$ in $H$, apply  the Moore ergodicity theorem, as in the proof of \cite[Proposition~2.1]{EOHecke}, to show that $\nu_n'$ is an ergodic $u(t)$-invariant measure, thus checking that the first hypothesis of \cite[Theorem~1.1]{MS} is satisfied.  We continue, in this way, following the proof of \cite[Proposition~2.1]{EOHecke} until the proof of Lemma \ref{le:algebraic} is complete.
\par

\par 
Finally the groups $x_n Hx_n^{-1}$ cannot all be equal to $H$ since
this would give $\gamma_n\in\Gamma$ so that $g_n H g_n ^{-1} = \gamma_n H \gamma_n ^{-1}$, contradicting the hypothesis that the curves $C_n$ are
pairwise different. We conclude 
that $H \subsetneq L$ and thus $L = G$ by Lemma~\ref{le:supergroup}:
\end{proof}
\par
\begin{lemma} \label{le:supergroup}
Let $(G,H)$ be as in Case One or Case Two.
If $L$ is a connected real Lie group with $H \subsetneq L \subset G$
and $\Gamma \cap L$ is a lattice in $L$, then $L = G$.
\end{lemma}
\par
\begin{proof} This is easily verified on the level of Lie algebras.
Since ${\rm Lie}(L)$ contains an element not in ${\rm Lie}(H)$, bracketing
with suitable elements of ${\rm Lie}(H)$ allows to produce 
a generating set of ${\rm Lie}(G)$.
\end{proof}

%%%%%%%%%%%%%%%%%%%%%%%%%%%%%%%%%%%%%%%%%%%%%%%%%%%%%%%%%%%%%
\section{The current of integration of a Shimura curve} \label{sec:current_of_integraion}
%%%%%%%%%%%%%%%%%%%%%%%%%%%%%%%%%%%%%%%%%%%%%%%%%%%%%%%%%%%%%

Any Shimura curve $C$, in fact any codimension one subvariety
of the Shimura surface $X$, defines a closed $(1,1)$-current on
$X$. On the other hand, the Shimura surfaces come with a
natural $(1,1)$-form, the K\"ahler form $\omega$. The aim of
this section is to translate the equidistribution result
(a convergence of measures) into a convergence statement for the
classes of these currents, suitably normalized. We start with the compact
case and explain at the end of this section the necessary modification
in the noncompact case. Recall that a $(1,1)$-current on a complex surface $X$ is a continuous linear functional on $A^{1,1}_c(X)$, the space of compactly supported $(1,1)$-forms on $X$.  This space $(A^{1,1}_c(X))^\vee$ contains both the complex curves $C\subset X$ and the smooth forms  $\eta\in A^{1,1}(X)$ by the formulas
$$
C \to (\alpha \to \int_C\alpha), \ \  \eta\to(\alpha\to  \int_X\eta\wedge\alpha) \ \ \text{ for all } \alpha\in A^{1,1}_c(X).
$$
The cohomology of $X$ can be computed either from the complex of forms or from the complex of currents.
Recall also that, if $X$ is K\"ahler$ $, $\omega$ denotes the K\"ahler form,  $\vol(X) =  \int_X \omega \wedge \omega$, 
that $\omega_C = \omega_X|_C$ is the K\"ahler form on $C$ and 
$\vol(C) = \int_{C} \omega_C$.
\par
\begin{prop} \label{prop:conv_currents}
Let  $X = \Gamma\backslash G/K$ be a smooth Shimura surface, let $g_n \in G$ be any sequence of points such that the 
 Shimura curves $C_n = \Gamma \backslash \Gamma g_n H/K$ 
are pairwise distinct. Then 
$$ C_n/\vol(C_n) \to \omega \quad \text{in} \quad A^{1,1}_c(X)^\vee, \quad \text{hence in} \quad H^{1,1}(X).$$ %\text{as currents on $X$.}$$
\end{prop}
\par
This and the finite-dimensionality of the Picard group allows
to deduce our main result.
\par
\begin{cor} \label{cor:C2}
Let  $X = \Gamma\backslash G/K$ be a compact,  smooth Shimura surface, let $g_n \in G$ be any sequence of points such that 
 the Shimura curves $C_n = \Gamma \backslash g_m H$/K  
are pairwise distinct. Then 
$$C_n^2 \sim \vol(\Gamma \backslash \Gamma g_n H)^2 \quad \text{for} \quad n \to \infty.$$
In particular for any $M$, there are only finitely many Shimura curves
$C$ on $X$ with $C^2 < M$.
\end{cor}
\par
\begin{proof}
For the first statement, 
fix a basis $\gamma_0 = \omega, \gamma_1,\ldots,\gamma_s$ of
$H^{1,1}(X)$. Taking $\gamma_i$ for $i>1$ orthogonal to $\gamma_0$, we may
suppose that the dual basis is $\lambda^{-1} \omega = \gamma_0^\vee, 
\gamma_1^\vee,\ldots,\gamma_s^\vee$ for some $\lambda \in \CC$, in fact
$\lambda = \int_X \omega\wedge\omega = \vol(X)$. If $C$ is a curve in $X$, 
thus representing a $(1,1)$-class, the Poincar{\'e} dual is represented by
$$ {\rm PD}(C) = \sum_{i=0}^s \bigl(\int_C \gamma_i \bigr)\,\, \gamma_i^\vee.$$
Consequently, by Proposition~\ref{prop:conv_currents}
\ba
&\frac{1}{A_n^2} C_n \cdot C_n = \frac{1}{A_n^2} \int_{C_n} {\rm PD}(C_n)
=  \sum_{i=0}^s   \Bigl(\frac{1}{A_n}\int_{C_n} \gamma_i \Bigr) \Bigl(\frac{1}{A_n}\int_{C_n} \gamma_i^\vee \Bigr) \\
&\longrightarrow \sum_{i=0}^s  \Bigl(\int_X \omega\wedge\gamma_i  \Bigr) \Bigl( \int_X \omega\wedge\gamma_i^\vee\Bigr)\,
= \, \lambda^2.
\ea
\par
The second statement follows from the first and from Corollary~\ref{cor:vol}.
\end{proof}
\par
\medskip
\paragraph{\bf Integrating on the projectivized tangent bundle.}
We now prepare for the proof of Proposition~\ref{prop:conv_currents}.
For this purpose we work on the universal cover $\tX = G/K$ of $X$.
First of all, for any (two-dimensional) 
K\"ahler manifold $X$ here is a natural map 
$$\PP T\tX \to \Lambda_{1,1} T\tX = (\Lambda^{1,1} T^*\tX)^\vee$$
defined pointwise at any $x\in \tX$
by  $[v] \mapsto v \wedge \bar{v}/|v|^2$ for $v \in T_x\tX \setminus \{0\}$.
Dually, an element $\alpha \in (\Lambda^{1,1} T^*\tX)$ defines a real-valued function 
$$\varphi_\alpha: \PP T \tX \to \RR, \quad \varphi_\alpha ([v]) = 
\alpha \Bigl( \frac{v \wedge \bar{v}}{|v|^2}\Bigr).$$
Using this map we can write the intersection with $\alpha$
as an integral of a real-valued function against the volume form of $\PP T X$.   In Case Two $\PP T\tX = G/K'$ is a homogeneous space with an invariant volume, where $K' =  \U(1) \times\U(1)$.  In Case One we will need to pass to a $G$-invariant real sub-bundle of $\PP T \tX$ also of the form $G/K'$ for $K' = \U(1)$. 
\par
We start with Case Two. Recall that we scaled the K\"ahler form $\omega$ so that $\vol(X) = \int_X \omega \wedge \omega $.
\par
\begin{lemma} \label{le:omegaalpha}
Let $X$ be a two-dimensional K\"ahler manifold, choose a two from $\eta$ on $\PP TX$ that restricts 
to the area form $\eta_x$ of each fiber $\PP T_x X$, $x\in X$, scaled to give total area one to each fiber.  
Then, for all $(1,1)$-forms $\alpha$ on $X$ and for each $x\in X$ we have  
$$(\omega \wedge \alpha)_x = \Bigl(\int_{\PP T_x X} \varphi_\alpha \eta_x \Bigr) 
(\omega \wedge \omega)_x.$$
Therefore we have
$$ \int_X \omega\wedge\alpha = \int_{\PP T X}\varphi_\alpha \  \eta\wedge\omega\wedge\omega,$$
where we have written simply $\omega$ for the pull-back to $\PP T X$ of the form $\omega$ on $X$.
\end{lemma}
\par
\begin{proof}
In suitable local coordinates at $x$, the K\"ahler form at $x$ is 
$\omega_x = \frac{i}{2} (dz_1\wedge d\bar{z}_1 + dz_2\wedge d\bar{z}_2) $.   
Writing  $\alpha = \frac{i}{2} \sum \alpha_{i\bar{j}} dz_i\wedge d\bar{z}_j$, we have (suppressing the factors of $\frac{i}{2}$)
 
$$(\omega \wedge \alpha)_x = (\alpha_{1,\bar{1}} + \alpha_{2,\bar{2}})
(dz_1\wedge d\bar{z}_1\wedge  dz_2\wedge d\bar{z}_2) =  \frac{\alpha_{1,\bar{1}} + \alpha_{2,\bar{2}}}{2}(\omega \wedge \omega)_x.$$
On the other hand, if we let $e_1,e_2$ denote the basis for $T_xX$ dual to $dz^1,dz^2$, and write $v = v_1 e_1 + v_2 e_2\in T_xX$, the first factor of the right hand side is
\ba
& \int_{\PP^1} \alpha\Bigl( \frac{(v_1e_1+v_2e_2)\wedge\overline{(v_1e_1+v_2e_2)}}{|v_1|^2 + |v_2|^2} \Bigr) \eta_x \\
&= \alpha_{1\bar{1}} \int_{\PP^1} \frac{|v_1|^2}{|v_1|^2 + |v_2|^2}\eta_x + \alpha_{2\bar{2}} \int_{\PP^1} \frac{|v_2|^2}{|v_1|^2 + |v_2|^2}\eta_x +  \int_{\PP^1} \frac{2\Im(\alpha_{1\bar{2}}\, v_1 \bar{v}_2)}{|v_1|^2 + |v_2|^2}\eta_x.
\ea
By symmetry reasons the last integral vanishes and the first two are equal and add to $1$, hence the first statement of the  lemma. The second follows from the first and Fubini's Theorem.
\end{proof}
\par
\begin{remark}
The first statement in the Lemma is equivalent to the well-known fact in linear algebra that the trace of a Hermitian matrix equals the average value over the unit sphere of the associated Hermitian form.
\end{remark}
\par
\begin{cor}\label{cor:omegaalphatwo}
It $X$ is a Shimura surface covered by the ball, then for all $(1,1)$-forms $\alpha$ on $X$ we have 
$$\int_X \omega\wedge\alpha = \int_{\PP T X} \varphi_\alpha\ d\nu_{\Gamma \backslash G/K'}$$
where $\nu_{\Gamma \backslash G/K'}$ is the volume form on $\PP T X$ introduced above.
\end{cor}
\begin{proof}
If $\tX =\BB^2 =  G/K$, then $\eta\wedge\omega\wedge\omega$ in Lemma~\ref{le:omegaalpha} is a 
$G$-invariant volume form on $\PP T \tX$. Moreover, $\omega$ and $\eta$ have been scaled to give
the correct normalization.
\end{proof}
\par
Now we address the corresponding statement in Case One. 
If the Shimura surface $X$ is covered by $\HH^2$, then $\PP T\tX$ is no longer a homogeneous space for $G$, but it has some natural homogeneous sub-bundles.  Equivalently, the action of $K$ on $\PP T_x \tX\cong \PP^1$  is not transitive, but has some distinguished orbits:  two zero-dimensional orbits, corresponding to the tangents to the two factors of $\HH^2$, and an orbit of real dimension one corresponding to the graphs of isometries between the two factors.  Explicitly, if we choose coordinates $z_1,z_2$ a above, this time adapted to the product structure of $\tX$, and with dual basis $e_1,e_2$ each tangent to one of the factors, and writing $v = v_1e_1 + v_2e_2$ as above, the action of $K\cong\U(1)\times \U(1)$ on $\PP T_x \tX \cong\PP^1$ leaves invariant the points with homogeneous coordinates $(1:0)$ and $(0:1)$ and the real submanifold $\{(v_1:v_2):|v_1| = |v_2|\} = \{(1:e^{i\theta})\}\cong S^1$.  
\par
Let us call this submanifold $\sS T_x \tX$ and let $\sS T\tX\cong G/K'$ denote the corresponding bundle 
over $\tX\cong G/K$ with fiber $K/K'\cong \sS T_x \tX \cong S^1$.  Then a calculation just as in 
the proof of Lemma~\ref{le:omegaalpha} gives us:
\par
\begin{lemma}\label{le:omegaalphaone}
Let $X$ be a Shimura surface covered by $\HH^2$,  choose a one form $\eta$ on $\sS TX$ that restricts to the angle form $\eta_x = d\theta$ of each fiber $\sS T_x X$,  scaled to give total area one to each fiber.  
Then, for any $(1,1)$ form $\alpha$ on $X$ and for each $x\in X$ we have  
$$(\omega \wedge \alpha)_x = \Bigl(\int_{\sS T_x X} \varphi_\alpha \eta_x \Bigr) 
(\omega \wedge \omega)_x.$$
Therefore we have
$$ \int_X \omega\wedge\alpha = \int_{\sS T X}\varphi_\alpha \  \eta\wedge\omega\wedge\omega = \int_{\sS T X}\varphi_\alpha \  d\nu_{\Gamma \backslash G/K'},$$
where $\nu_{\Gamma \backslash G/K'}$ is the volume form on $\sS T X$ introduced above. %where $d\mu$ is the $G$-invariant volume form on $\tX$.
\end{lemma}
\par
\begin{proof}[Proof of Proposition~\ref{prop:conv_currents}]
To show convergence in $H^{1,1}(X)$ it suffices to show that 
$$\frac1{\vol(C_n)} \int_{C_n} \alpha \to \int_X \omega \wedge \alpha$$
for any $\alpha \in H^{1,1}(X)$. In Case Two, by Corollary~\ref{cor:omegaalphatwo}
it suffices to show that 
$$\frac1{\vol(C_n)} \int_{C_n} \alpha \to \int_{\PP T X} \varphi_\alpha d\nu_{\Gamma \backslash G/K'}.$$
A local verification, just using the definition of $\varphi_\alpha$
and the fact that $\nu_{C_n,K'}$ was defined to give measure one to the fibers $K/K'$
implies that
$\int_{C_n} \alpha = \int_{\PP T C_n} \varphi_\alpha d\nu_{C_n,K'}$. Since
$\nu_{C_n,K'}$ is supported on $ {\PP T C_n} \subset {\PP T X}$, it is
thus sufficient to show that
$$ \int_{\PP T X} \varphi_\alpha d\nu_{C_n,K'} \to \int_{\PP T X} \varphi_\alpha d\nu_{\Gamma \backslash G/K'}.$$
We have reformulated our claim in terms of a convergence of measures, integrating
against a globally defined function $\varphi_\alpha$.  Proposition~\ref{prop:equidistK} completes the proof.
In Case One, the proof is the same, replacing  ${\PP T X}$ by ${\sS T X}$ and the the reference to 
Corollary~\ref{cor:omegaalphatwo} by Lemma~\ref{le:omegaalphaone}.
\end{proof}
\par
\paragraph{\bf The non-compact case.} Recall that we denoted by $Y$
a minimal resolution of the singularities of the Baily-Borel compactification
$X^\BB$. By \cite[Theorem~3.1 and Proposition~1.1]{mumford77} the K\"ahler class $\omega$ extends to a closed current on $Y$. 
Moreover $\omega \in \pi^*H^2(X^\BB,\QQ)$ by \cite[Proposition~3.4~(b)]{mumford77}. The statement of Proposition~\ref{prop:conv_currents} now reads
$$p_{B^\perp} (C_n)/\vol(C_n) \to \omega \quad \text{in} \quad  \pi^*H^2(X^\BB,\QQ),$$
where $p_B^\perp$ is the orthogonal projection onto the complement of $B$.
The same proof as above works. In order to show the analog 
$$(p_{B^\perp} C_n)^2 \sim \vol(\Gamma \backslash \Gamma g_n H)^2 \quad \text{for} \quad n \to \infty$$
of Corollary~\ref{cor:C2} we apply the Poincar\'e 
duality to $\pi^*H^2(X^\BB,\QQ)$. Since this is a perfect pairing, 
the proof of Corollary~\ref{cor:C2} applies without changes:
\par
\begin{theorem} \label{thm:main_noncompact}
For $X$ as above and for any real number  $M$ there are only finitely many  Shimura curves 
$C$ on  $X$ with $(p_{B^\perp} C)^2 < M$. 
\end{theorem}
In particular, for the collection of compact Shimura curves in $X$ we obtain Theorem \ref{thm:intro_main}.

\bibliographystyle{halpha}
\bibliography{bib_SI}

\end{document}